\documentclass[11pt,reqno]{amsart}
\usepackage{color}
\usepackage[active]{srcltx}
\usepackage{a4wide}
\usepackage{amssymb, amsmath, mathtools}
\usepackage{graphicx}
\usepackage{float}
\usepackage[active]{srcltx}
\usepackage[ansinew]{inputenc}
\usepackage{hyperref}

\theoremstyle{plain}
\newtheorem{theorem}{Theorem}[section]
\newtheorem{maintheorem}{Theorem}
\newtheorem{lemma}[theorem]{Lemma}
\newtheorem{proposition}[theorem]{Proposition}

\theoremstyle{remark}

\newtheorem{remark}[theorem]{Remark}

\numberwithin{equation}{section}

\newcommand{\NN}{{\mathbb{N}}}
\newcommand{\ZZ}{{\mathbb{Z}}}
\newcommand{\RR}{{\mathbb{R}}}
\newcommand{\EU}{{\mathbb{S}}}

\newcommand{\In}{{\text{In}}}
\newcommand{\Out}{{\text{Out}}}
\newcommand{\Fix}{{\text{Fix}}}
\newcommand{\loc}{{\text{loc}}}

\newcommand{\dpt}{\displaystyle}
\newcommand{\zg}[1]{\langle\gamma_{#1}\rangle}

\begin{document}

\title[The role of the saddle-foci on the structure of a Bykov attracting set]{The role of the saddle-foci on the structure of a Bykov attracting set}
\author[M. Bessa]{M\'ario Bessa}
\address{M\'ario Bessa \\ Departamento de Matem\'atica da Univ. da Beira Interior\\ Rua Marqu\^es d'\'Avila e Bolama \\ 6201-001 Covilh\~a\\ Portugal}
\email{bessa@ubi.pt}

\author[M. Carvalho]{Maria Carvalho}
\address{Maria Carvalho \\ Centro de Matem\'atica da Univ. do Porto \\ Rua do Campo Alegre, 687 \\ 4169-007 Porto \\ Portugal}
\email{mpcarval@fc.up.pt}

\author[A. Rodrigues]{Alexandre A. P. Rodrigues}
\address{Alexandre Rodrigues \\ Centro de Matem\'atica da Univ. do Porto \\ Rua do Campo Alegre, 687 \\ 4169-007 Porto \\ Portugal}
\email{alexandre.rodrigues@fc.up.pt}
\thanks{The authors are grateful for the careful reading and the helpful comments and suggestions that much improve this manuscript.}
\date{\today}




\maketitle

\begin{abstract}
We consider a one-parameter family $(f_\lambda)_{\lambda \, \geqslant \, 0}$ of symmetric vector fields on the three-dimensional sphere $\EU^3\subset\RR^{4}$ whose flows exhibit a heteroclinic network between two saddle-foci inside a global attracting set. More precisely, when $\lambda = 0$, there is an attracting heteroclinic cycle between the two equilibria which is made of two $1$-dimensional connections together with a $2$-dimensional sphere which is both the stable manifold of one saddle-focus and the unstable manifold of the other. After slightly increasing the parameter while keeping the $1$-dimensional connections unaltered, the two-dimensional invariant manifolds of the equilibria become transversal, and thereby create homoclinic and heteroclinic tangles. It is known that these newborn structures are the source of a countable union of topological horseshoes, which prompt the coexistence of infinitely many sinks and saddle-type invariant sets for many values of $\lambda$. We show that, for every small enough positive parameter $\lambda$, the stable and unstable manifolds of the equilibria and those infinitely many horseshoes are contained in the global attracting set of $f_\lambda$. Moreover, we prove that the horseshoes belong to the heteroclinic class of the equilibria. In addition, we verify that the set of chain-accessible points from either of the saddle-foci is chain-stable and contains the closure of the invariant manifolds of the two equilibria.
\end{abstract}

\setcounter{tocdepth}{1}

\section{Introduction}\label{intro}

This work starts with an autonomous $(\ZZ_2\times\ZZ_2)-$equivariant vector field $f_0$ on the sphere $\EU^3 \subset \RR^4$ exhibiting an attracting heteroclinic network between two saddle-foci $\sigma_1$ and $\sigma_2$ which
belong to a flow-invariant submanifold whose existence is a consequence of the symmetry. This network consists of a two-dimensional sphere connecting $\sigma_2$ and $\sigma_1$ and two one-dimensional connections from $\sigma_1$ to $\sigma_2$. See Figure~\ref{Vegter1} (left). Afterwards, $f_0$ is perturbed generating a smooth one-parameter family $(f_\lambda)_{\lambda \, \geqslant \, 0}$ of vector fields subject to several conditions we will summarize in Section~\ref{setting}. Although small, the perturbation breaks part of the symmetry. More precisely, when $\lambda > 0$, the one-dimensional connections persist while the two-dimensional invariant manifolds, which coincided in the beginning, now intersect transversely. Yet, the closeness of $f_\lambda$ when $\lambda >0$ to the symmetric context somehow simplifies our study due to constrains it induces on the geometry of the invariant manifolds of the equilibria, allowing some control of their relative positions. The aim of this work was to understand how the configuration of these invariant manifolds enables hyperbolic saddle-type behavior to coexist with infinitely many sinks and chaotic attractors. This might explain, though partially, the origins of the chaos that computer simulations of these kind of systems very often disclose, and the role of the stable and unstable manifolds of both the equilibria in the bifurcation phenomena that are known to occur when the parameter $\lambda$ is changed.

\begin{figure}[h]
\begin{center}
\includegraphics[height=7cm]{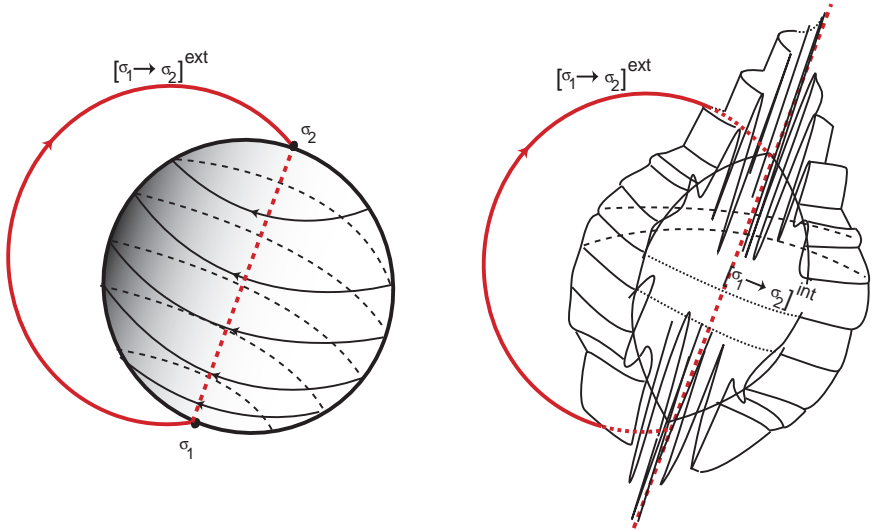}
\end{center}
\caption{\small The invariant manifolds of the saddle-foci for  $\lambda=0$ (left) and $\lambda > 0$ small (right). The latter image was inspired by Figure 4 of \cite{BrVg75}.}
\label{Vegter1}
\end{figure}

We briefly recall in Section~\ref{overview} the intervention, in the bifurcation process that $(f_\lambda)_{\lambda \, \geqslant \, 0}$ is undergoing, of both the symmetry assumptions and the complex nature of eigenvalues of the equilibria, as well as of the transversal intersection (when $\lambda > 0$) of the two-dimensional invariant manifolds of the equilibria. They are responsible for the appearance of infinitely many topological horseshoes and a profusion of homoclinic tangencies, resulting in the outbreak of Newhouse phenomena and the genesis of huge sets of sinks. See Figure~\ref{Vegter1} (right).
This pattern agrees with the known models of quasi-stochastic attractors, found analytically or evinced through computer simulations (cf. \cite{Afraimovich83}). Up to our knowledge, the two-dimensional invariant manifolds associated to the heteroclinic phenomena have been considered on previous references only locally, near an equilibrium or inside a tubular neighborhood of the cycle. In this work we focus on the behavior of the entire two-dimensional manifolds of the equilibria as organizers of the dynamics. However, in order to be able to compute the dynamical data, we have to restrict our analysis to a model of such a family of vector fields, whose properties will be specified in Section~\ref{modelo}.

The paper is organized as follows. After introducing the model of the one-parameter family of vector fields we work with, and discussing its main properties, 
we relate the countable union of topological horseshoes that are known to appear in this setting with the heteroclinic class of the two equilibria, and elucidate their relevance in the bifurcation portrait. Noticing that the global attracting set of $f_\lambda$ may be non transitive due to the presence of proper attractors, we describe in Section~\ref{Proof of Theorem C} a dynamically relevant chain-stable set. For the reader's convenience, we have compiled at the end of the paper a list of definitions in a short glossary.

\section{The setting}\label{setting}

Let $\left(f_\lambda\right)_{\lambda \in \mathbb{R}}$ be a one-parameter family of vector fields in $\mathfrak{X}^r(\EU^3)$, $r \geqslant 3$, whose flows are given by the solutions $\left(\varphi_\lambda(t,x)\right)_{t\,\in\,\mathbb{R}, \,x \,\in \,\EU^3}$ of the differential equations $\dot{x}= f_\lambda(x)$. We will enumerate the main assumptions concerning symmetry, chirality, the complex eigenvalues at the saddle-foci and the intersection of the two bi-dimensional invariant manifolds of these equilibria. We refer the reader to the Appendix~\ref{Definitions} for precise definitions.
\bigskip
\subsection{Main assumptions}\label{sse:main-hyp}
We will request that the organizing center at $\lambda=0$ satisfies the following assumptions:
\bigskip
\begin{enumerate}
\item[(\textbf{P1})] The vector field $f_0$ is equivariant under the action of $ \ZZ_2 \oplus \ZZ_2$ on $\EU^3$ induced by the linear maps $\gamma_1(x_1,x_2,x_3,x_4)=(-x_1,-x_2,x_3,x_4)$ and $\gamma_2(x_1,x_2,x_3,x_4)=(x_1,x_2,-x_3,x_4)$ on $\RR^4$.
\bigskip

\item[(\textbf{P2})] The set $\Fix(\ZZ_2 \oplus \ZZ_2)=\left\{Q \in \EU^3:\gamma_1(Q)=\gamma_2(Q)= Q \right\}$ reduces to the equilibria $\sigma_1=(0,0,0,1)$ and $\sigma_2=(0,0,0,-1)$, which are hyperbolic saddle-foci and whose eigenvalues are respectively $-C_{1} \pm \alpha_{1}i$ and $E_{1}$, where $\alpha_{1} \neq 0$ and $C_{1}>E_{1}>0$, and $E_{2} \pm \alpha_{2}i$ and $-C_{2}$, where $\alpha_{2} \neq 0$ and $C_{2}>E_{2}>0$. We will assume that $\alpha_1=\alpha_2=1$.
\bigskip

\item[(\textbf{P3})] The flow-invariant circle $\Fix(\ZZ_2 \zg{1})=\left\{Q \in \EU^3:\gamma_1(Q) = Q\right\}$ consists of $\sigma_1$ and $\sigma_2$, and two heteroclinic trajectories $[\sigma_1 \rightarrow \sigma_2]^{ext}$ and $[\sigma_1 \rightarrow \sigma_2]^{int}$ from $\sigma_1$ to $\sigma_2$ whose union will be denoted by $[\sigma_1 \rightarrow \sigma_2]$.
\bigskip

\item[(\textbf{P4})] The $f_0-$invariant sphere $\Fix(\ZZ_2 \zg{2})=\left\{Q \in \EU^3:\gamma_2(Q) = Q \right\}$ is made of $\sigma_1$ and $\sigma_2$ and a two-dimensional heteroclinic connection from $\sigma_2$ to $\sigma_1$.
\end{enumerate}
\bigskip

\begin{enumerate}
\item[(\textbf{P5})] The saddle-foci $\sigma_1$ and $\sigma_2$ have the same chirality (cf. Appendix~\ref{chirality_def}).
\end{enumerate}
\bigskip

The sphere $\Fix(\ZZ_2 \zg{2})$ divides $\EU^3$ into two connected components (inner and outer) and $[\sigma_1 \rightarrow \sigma_2]\setminus\{\sigma_1, \sigma_2\}$ has also two connected components, one on each connected component of $ \EU^3\setminus \Fix(\ZZ_2 \zg{2})$. The set $[\sigma_1 \rightarrow \sigma_2]^{int}$ (respectively, $[\sigma_1 \rightarrow \sigma_2]^{out}$) is precisely the connection lying in the inner (respectively, outer) component of
$\EU^3\setminus \Fix(\ZZ_2 \zg{2})$. The two equilibria, the two trajectories referred to in (P3) and the two-dimensional heteroclinic connection from $\sigma_2$ to $\sigma_1$ mentioned in (P4) build a \emph{heteroclinic network} we will denote hereafter by $\Sigma_0$. It is illustrated in Figure~\ref{Vegter1} (left).

Denote by $\mathfrak{X}^r_{\text{Byk}}(\EU^3)\subset \mathfrak{X}^r(\EU^3)$ the set, endowed with the $C^r$--Whitney topology, of $\zg{1}$--equivariant $C^r$ vector fields on $\EU^3$, $r \geqslant 3$, satisfying (P1)--(P5). Consider $f_0 \in \mathfrak{X}^r_{\text{Byk}}(\EU^3)$ embedded in a {continuous} one-parameter family of vector fields
$$ \lambda \in \mathbb{R} \,\,\mapsto \,\, f_\lambda\equiv f(x, \lambda):\EU^3 \rightarrow T\EU^3 \,\in \,\mathfrak{X}^r_{\text{Byk}}(\EU^3)$$
breaking the $\zg{2}$--equivariance at $\lambda = 0$ {in a generic way}. The next three assumptions are the properties we will also need to be valid for the family $\left(f_\lambda\right)_{\lambda \in \RR}$.
\medbreak

\begin{enumerate}
\item[(\textbf{P6})] For every $\lambda\in \RR$, the map $f_\lambda$ is a $\zg{1}$--equivariant $C^r$ vector field,
$r \geqslant 3$.
\end{enumerate}
\medbreak

\begin{enumerate}
\item[(\textbf{P7})] For $\lambda \neq 0$, the two-dimensional manifolds $W^u_\lambda(\sigma_2)$ and $W^s_\lambda(\sigma_1)$ intersect transversely (abbreviated to $W^u_\lambda(\sigma_2)\cap W^s_\lambda(\sigma_1)$) at two trajectories $\beta_1$ and $\beta_2$, called \emph{primary links} or \emph{$0$--pulses}.
\end{enumerate}
\medbreak

\begin{enumerate}
\item[(\textbf{P8})] For every $\lambda \geqslant 0$, the transitions along the connections $[\sigma_1 \rightarrow \sigma_2]$ and $[\sigma_2 \rightarrow \sigma_1]$ are given, in local coordinates, by the Identity map and, up to high order terms in the variable $y$, the transformation $$(x,y)\mapsto (x,y+\lambda \sin(x))$$ respectively. This assumption will be detailed in \S\ref{transitions} and \S\ref{modelo}.
\end{enumerate}
\medbreak

A few comments are in order. Since the equilibria lie on $\Fix(\ZZ_2 \zg{1})$ and are hyperbolic, they persist for small $\lambda>0$, and still satisfy the properties (P2) and (P3).
Besides, $\sigma_1$ and $\sigma_2$, the primary links and the connections $[\sigma_1 \rightarrow \sigma_2]^{ext}$ and $[\sigma_1 \rightarrow \sigma_2]^{int}$ form a heteroclinic network that we denote by $\Sigma^\star_\lambda$. In what follows, $\Gamma $ will stand for a \emph{Bykov cycle} in $\Sigma^\star_\lambda$ with the connections mentioned in (P3) and (P7). According to \cite{LR}, the existence of the  network $\Sigma^\star_\lambda$  implies the presence of a bigger network $\Sigma_\lambda$ (not yet entirely understood) containing infinitely many copies of a Bykov cycle. More precisely, beyond the original primary links $\beta_1$ and $\beta_2$, there exist infinitely many subsidiary heteroclinic connections (transversal or not) turning around the original Bykov cycle. Since the complete description of $\Sigma_\lambda$ is, at the moment, unreachable, we concentrate our attention on $\Sigma^\star_\lambda$.

The set of vector fields $f_0 \in \mathfrak{X}^r_{\text{Byk}}(\EU^3)$, $r \geqslant 3$, which satisfy the conditions (P1)--(P5) contains a non-empty open subset of families $\left(f_\lambda\right)_{\lambda\in \RR}$ for which the assumptions (P6)--(P8) hold (cf. \cite{RodLab}). In particular, there are such families of polynomial vector fields.

\section{Overview}\label{overview}

Let $\left(f_\lambda\right)_{\lambda \geqslant 0}$ be a one-parameter family of vector fields in $\mathfrak{X}_{Byk}^3(\EU^3)$ satisfying (P1)--(P8). When $\lambda=0$, the heteroclinic connections in the network $\Sigma_0$ are contained in fixed point subspaces satisfying the hypothesis (H1) of \cite{KM1}. Due to the inequality $C_1 C_2 >E_1 E_2$, the stability criterion in \cite{KM1} may be applied to $\Sigma_0$, and so there exists a three-dimensional normally hyperbolic invariant manifold where every solution eventually enters. Therefore, we may find an open neighborhood $\mathcal{U}^0$ of the heteroclinic network $\Sigma_0$ (whose closure is the set $V^0$) having its boundary transverse to the flow of the vector field $f_0$ and such that every solution starting in $\mathcal{U}^0$ remains in it for all positive time and is forward asymptotic to $\Sigma_0$. The global attracting set of $f_0$, namely
$$\mathcal{A}_0 =[\sigma_1 \rightarrow \sigma_2] \cup W^u_0(\sigma_2)=[\sigma_1 \rightarrow \sigma_2] \cup W^s_0(\sigma_1)$$
is made of $\sigma_1, \sigma_2$ and the two connections $[\sigma_1 \rightarrow \sigma_2]^{ext}$ and $[\sigma_1 \rightarrow \sigma_2]^{int}$, together with a sphere which is both the stable manifold of $\sigma_1$ and the unstable manifold of $\sigma_2$ (see the left part of Figure \ref{Vegter1}).

As the transverse property is robust, for small enough $\lambda >0$,
the neighborhood $V^0$ is also positively invariant by $f_\lambda$ (that is, every solution of the vector field eventually enters this region and never leaves it in the future) and the flow of $f_\lambda$ still has a three-dimensional normally hyperbolic invariant manifold transverse to the flow, inwardly oriented and containing the network $\Sigma_\lambda$. Due to the transversality of the flow to the boundary of some compact neighborhood $V^0$ of $\mathcal{A}_0$ such that $\varphi_\lambda(t, V^0)$ is contained in the interior of $V^0$ for all $t \geqslant 0$, when $\lambda>0$ is small enough the flow of $f_\lambda$ still has a global attracting set, namely
\begin{equation}\label{A_lambda}
\mathcal{A}_\lambda = \bigcap_{t \geqslant 0} \,\,\varphi\left(t, V^0\right).
\end{equation}

We observe that, as the initial attracting set $\mathcal{A}_0$ is asymptotically stable and contains all the invariant manifolds of the equilibria, there exists $\lambda_0>0$ such that, for every $0\leqslant \lambda\leqslant \lambda_0$, those invariant manifolds associated to $f_\lambda$ are still in the interior of $V^0$. Using their flow invariance, we conclude that:
\begin{equation}
\label{union_final}
\overline{W^s_\lambda(\sigma_1)\cup W^u_\lambda(\sigma_1)\cup W^s_\lambda(\sigma_2)\cup W^u_\lambda(\sigma_2)}\subset \mathcal{A}_\lambda.
\end{equation}

In what follows, $\mathcal{A}_\lambda$ will be called \emph{a Bykov attracting set}, which may have infinitely many connected components.
\medbreak
In \cite{ACL05, Bykov00, LR, LR2015}, it was proved that there exists an open subset $\mathcal{C}$ of $\mathfrak{X}_{Byk}^3(\EU^3)$ containing $f_0$ and such that the following properties hold for every family of vector fields $\left(f_\lambda\right)_\lambda$ subject to the conditions (P1)--(P8):

\bigskip

\noindent \textbf{1}. For $n \in \NN$ and $\lambda > 0$, any tubular neighborhood of a Bykov cycle in $\Sigma^\star_\lambda$ contains infinitely many $n$--pulse (cf. Appendix~\ref{pulse}) heteroclinic connections $[\sigma_2\to\sigma_1]$.
\bigskip

\noindent \textbf{2}. For any tubular neighborhood $\mathcal{T}$ of $\Sigma_\lambda^\star$ and every cross-section to the flow $S_q\subset \mathcal{T}$ at a point  $q$ in $[\sigma_2\rightarrow\sigma_1]$, the first return map to $S_q$ has a countable family of compact invariant sets $\left(\Lambda_{\lambda,m}\right)_{m \,\in \,\mathbb{N}}$ in each of which the dynamics is conjugate to a full shift over a finite number of symbols. The union $\Lambda_\lambda$ of these sets accumulates on $\Sigma^\star_\lambda$ while the number of symbols coding the first return map to $S_q$ tends to infinity.  Moreover, for every $m \in \mathbb{N}$, the topological horseshoe $\Lambda_{\lambda,m}$ is uniformly hyperbolic \cite{ACL05}.
\bigskip

\noindent \textbf{3}. At any cross-section $S_q\subset \mathcal{T}$, the set of initial conditions in $S_q$ that do not leave $\mathcal{T}$ by the flow, for all future and past times, is precisely $\Lambda_\lambda$. In particular, any tubular neighborhood $\mathcal{T}$ of a Bykov cycle $\Gamma$ inside $\Sigma^\star_\lambda$ contains points not lying in $\Gamma$ whose trajectories remain in $\mathcal{T}$ for all time.
\bigskip

\noindent \textbf{4}. There is a strictly decreasing sequence $\left(\lambda_i\right)_{i \,\in\, \mathbb{N}}$ of positive real numbers converging to $0$ and such that, for any $\lambda>\lambda_i$, there are two 1--pulse heteroclinic connections for the flow of $\dot{x}=f_\lambda(x)$ that collapse into a 1--pulse heteroclinic tangency at $\lambda=\lambda_i$, and then disappear for $\lambda<\lambda_i$.
\bigskip

\noindent \textbf{5}. For each $i \in \mathbb{N}$, there is a sequence of parameter values $\left(\lambda_{ij}\right)_{j \,\in \,\mathbb{N}}$ accumulating at $\lambda_i$ for which the corresponding vector field has a $k$--pulse heteroclinic tangency for every $k \in \mathbb{N}$.
\bigskip

\noindent \textbf{6}. As $\lambda$ decreases to zero, some of the saturated horseshoes associated to these heteroclinic tangencies disappear, while we witness the emergence of Newhouse phenomena and several $k$--pulses heteroclinic connections are destroyed through saddle-node type bifurcations.

\bigskip

In order to improve the readability of the paper, we have gathered in the following table the most important symbols used in the text.
\medskip
\small
\begin{table}[htb]
\begin{center}
\begin{tabular}{|c|l|} \hline
Notation  & \hspace{2.5cm} Meaning   \\
\hline
& \\
 $\Sigma_0$ & Initial attracting heteroclinic network \\
 &\\
  \hline
  &\\
 $\Sigma_\lambda^\star$ & Network with four Bykov cycles  \\
 &\\  \hline
 &\\
  $\Sigma_\lambda$ & Bigger network with infinitely many pulses   \\
  &\\  \hline
  & \\
  $\mathcal{A}_\lambda$&  Global attracting set   \\
    & \\ \hline
  & \\
  ${\Lambda}_\lambda$& Union of countable many horseshoes   \\
    & \\

   \hline
\end{tabular}
\end{center}
\bigskip
\caption{Symbols used throughout the paper.}
\end{table}
\normalsize

\bigbreak
The transition from a simple steady state at $\lambda = 0$ to chaotic attractors when $\lambda > 0$ reminds of vector fields exhibiting instant chaos, as the ones presented in \cite{GW}. When $\lambda$ changes, we witness the creation or destruction of invariant pieces of the non-wandering set, as well as the loss of the hyperbolic properties at some of them. We expect to see the family of flows $(f_\lambda)_\lambda$ undergoing saddle-node bifurcations inducing the destruction of horseshoes \cite{Crov2002}, cascades of period-doubling \cite{YA} and the unfolding of homoclinic tangencies \cite{PT}. Besides, some of these bifurcations may happen simultaneously. Yet, the complete bifurcation diagram is at the moment out of reach.

\section{Main results}\label{main results}

Let $\left(f_\lambda\right)_{\lambda\in \RR}$ be a one-parameter family of vector fields in $\mathfrak{X}_{Byk}^3(\EU^3)$ satisfying the conditions (P1)--(P8). For every small enough $\lambda >0$, there exists a saddle-type set $\Lambda_\lambda$, invariant by the first-return map $\mathcal{P}_\lambda$ defined at an adequate cross-section $S_\lambda$, which is a countable union of horseshoes. Our first result, whose proof will be presented in Section~\ref{Proof of Theorem A}, states that the saturation $\widetilde{\Lambda_\lambda}$ by the flow is dynamically linked to the equilibria (the operator tilde is defined in Appendix~\ref{operador_tilde}).
In what follows, the solution by the vector field of a fixed point in $\Lambda_\lambda$ under the first return map to the cross section $S_q$  is called $1$-periodic orbit of $\widetilde{\Lambda_\lambda}$.

\begin{maintheorem}\label{thm:A}
Let $\left(f_\lambda\right)_\lambda$ be a one-parameter family of vector fields in $\mathfrak{X}_{Byk}^3(\EU^3)$ satisfying (P1)--(P8). Consider a tubular neighborhood $\mathcal{T}$ of the Bykov network $\Sigma^\star_\lambda$.
Then there exists $\lambda_0>0$ such that, for every $0 < \lambda \leqslant \lambda_0$, one has:
\medbreak
\begin{enumerate}
\medskip
\item  $\overline{\widetilde{\Lambda_\lambda}} \subset \overline{W^u_\lambda(\sigma_2)\cap W^s_\lambda(\sigma_1)}.$

\medskip
\item $\overline{W^s_\lambda(\widetilde{\Lambda_\lambda})}\subset\overline{W^s_\lambda(\sigma_1)}\,\,$ and $\,\,\overline{W^u_\lambda(\widetilde{\Lambda_\lambda})}\subset \overline{W^u_\lambda(\sigma_2)}.$
\medskip
\item The Lyapunov exponents at the $1$-periodic orbits of $\widetilde{\Lambda_\lambda}$ are uniformly bounded away from zero.

\end{enumerate}
\end{maintheorem}



Without much information about the global attracting set $\mathcal{A}_\lambda$, we decided to look for chain-accessible and chain-recurrent points. In spite of exploring approximate solutions instead of true trajectories, this strategy has the advantage of reasonably explaining some difficult aspects of the phase portrait of $f_\lambda$. Let $V^0$ be as in Section~\ref{overview} and denote by $\mathcal{B}_\lambda$ the set
\begin{equation}
\mathcal{B}_\lambda = \{P \in V^0: P\, \text{ is chain-accessible from } \,\sigma_2\}.
\end{equation}
When $\lambda=0$, one has $\mathcal{A}_0=\mathcal{B}_0$. For $\lambda > 0$, there may exist either attracting periodic orbits or other more complex proper attractors, whose presence in the phase space of $f_\lambda$ is hard to confirm, but the dynamical role of the set $\mathcal{B}_\lambda$ within the global attracting set is also worthwhile to unravel due to the following properties (to be proved in Section~\ref{Proof of Theorem C}).

\begin{maintheorem}\label{thm:C}
Under the hypotheses of Theorem \ref{thm:A}, there exists $\lambda_1 > 0$ such that, for every $0 < \lambda \leqslant \lambda_1$, the set $\mathcal{B}_\lambda$ is compact, forward invariant and chain-stable. Moreover, the union $\overline{W^u_\lambda (\sigma_2)}\,\cup\,\overline{W^s_\lambda (\sigma_1)}$ is contained in $ \mathcal{A}_\lambda \cap  \mathcal{B}_\lambda$.
\end{maintheorem}

\section{Transition maps}\label{se:local-coordinates}

In this section we will analyze the dynamics near the Bykov cycles of $\Sigma^\star_\lambda$ through local maps, after selecting appropriate linearizing coordinates in cylindrical neighborhoods of the saddle-foci $\sigma_1$ and $\sigma_2$ (see Figure~\ref{conservativo3}), as done in \cite{LR}. In those cylindrical coordinates $(\rho ,\theta ,z)$ the linearization of the dynamics at $\sigma_1 $ and $\sigma_2$ is specified by the following equations:
\begin{equation}\label{local map v}
\left\{\begin{array}{l}
\dot{\rho}=-C_1 \rho \\
\dot{\theta}= 1 \\
\dot{z}=E_1 z
\end{array}
\right.
\qquad\text{and}\qquad
\left\{\begin{array}{l}
\dot{\rho}=E_2 \rho \\
\dot{\theta}= 1 \\
\dot{z}=-C_2 z .
\end{array}
\right.
\end{equation}

\begin{figure}[h]
\begin{center}
\includegraphics[height=8.5cm]{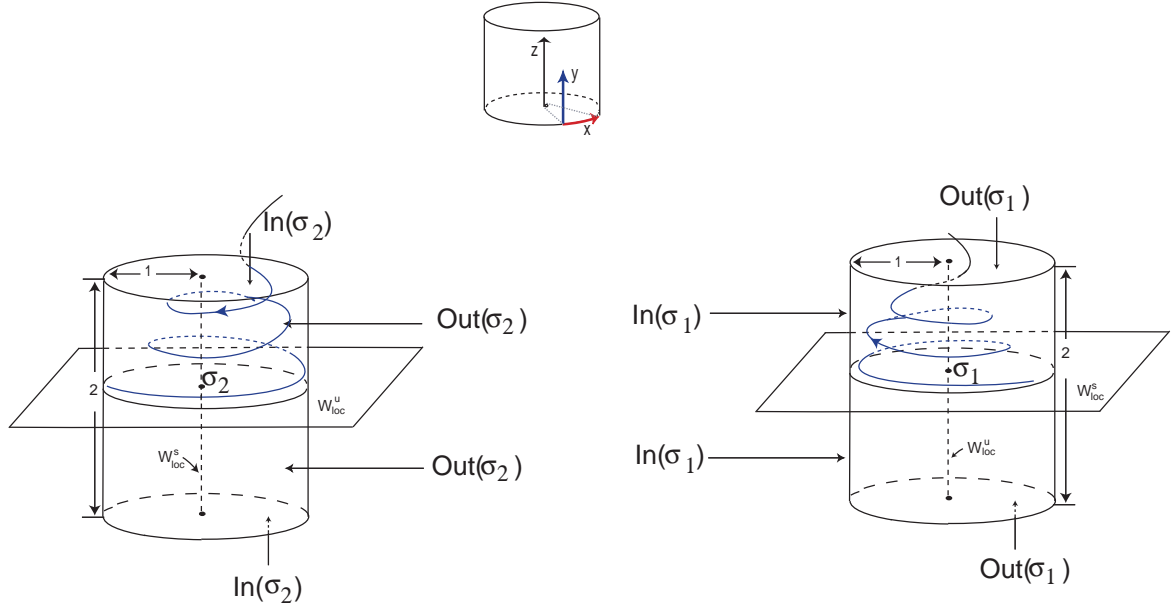}
\end{center}
\caption{Local cylindrical coordinates in $V_1$ and $V_2$.}
\label{conservativo3}
\end{figure}

Assume that those cylindrical neighborhoods of $\sigma_1 $ and $\sigma_2 $, which we call $V_1$ and $V_2$,  have base-radius $\varepsilon>0$ and height $2\varepsilon$. After a linear rescaling of the variables, we may suppose that $\varepsilon=1$. The boundaries of $V_1$ and $V_2$ have three components: the cylinder wall, parameterized in cylindrical coordinates $(r, \varphi, z) \equiv (1,x,y)$ by $x\in \,[0,2\pi[$ and $|y| \leqslant 1$; and two discs, the top and bottom of each cylinder, with polar parameterizations $(r,\varphi,\pm 1)$, where $0 \leqslant r \leqslant 1$ and $\varphi \in \,[0, 2\pi[$.

In $V_1$, we will use the following notation: $\In(\sigma_1)$ is the subset of the cylinder wall of $V_1$ consisting of points that enter $V_1$ in positive time; $\Out(\sigma_1)$ is the subset of the top and bottom of $V_1$ made of the points that leave $V_1$ in positive time; $\In^+(\sigma_1)$ is the upper part of the cylinder, parameterized by $(x,y)$ with $x\in [0, \, 2\pi)$ and $y\,\in\,(0,1)$; dually, $\In^-(\sigma_1)$ is its lower part, parameterized by $(x,y)$ with $x\in [0, \, 2\pi)$ and $y\,\in\,(-1,0)$. Similarly, we define the cross-sections for the linearization inside the neighborhood $V_2$ of $\sigma_2$. Observe that the flow is transverse to these cross-sections.


The local stable manifold of $\sigma_1$, say $W^s_{\lambda,\loc}(\sigma_1)$, corresponds precisely to the disk in $V_1$ parameterized by $z=0$. The local unstable manifold $W^u_{\lambda,\loc}(\sigma_1)$ of $\sigma_1$ is the $z$-axis in $V_1$, intersecting the top and bottom disks of this cylinder at their centers. A similar description of the local invariant manifolds of $\sigma_2$ is valid in $V_2$. 

\subsection{Local maps near the saddle-foci}\label{se:local saddle-foci}
Integrating \eqref{local map v}, we deduce that the trajectory of a point $(x,y) \in \In^+(\sigma_1)$ leaves $V_1$ through $\Out^+(\sigma_1)$ at the point
\begin{equation}\label{local_v}
\Phi^+_{1}(x,y)=\left(y^{\delta_1},\,\,-\frac{\log y}{E_1}+x \right)=(r,\varphi)
\end{equation}
where $\delta_1 =\frac{C_{1 }}{E_{1}} > 1$. Similarly, a point $(r,\varphi) \in \In^+(\sigma_2) \setminus W^s_{\lambda,\loc}(\sigma_2)$ leaves $V_2$ by $\Out^+(\sigma_2)$ at
\begin{equation}\label{local_w}
\Phi^+_{2}(r,\varphi )=\left(-\frac{\log r}{E_2}+\varphi,\,\,r^{\delta_2 }\right)=(x,y)
\end{equation}
where $\delta_2=\frac{C_{2 }}{E_{2}} >1$. Analogous expressions define the local maps $\Phi^-_{1}: \In^-(\sigma_1)\rightarrow \Out^-(\sigma_1)$ and $\Phi^-_{2}: \In^-(\sigma_2)\rightarrow \Out^-(\sigma_2)$. Finally, the map
$$\Phi_1\,: \In^+(\sigma_1) \cup \In^-(\sigma_1) \rightarrow \Out(\sigma_1)$$
is just given by $\Phi_1^+$ in $\In^+(\sigma_1)$ and $\Phi_1^-$ in $\In^-(\sigma_1).$ Similarly, we gather in $\Phi_2$ both maps $\Phi_2^+$ and $\Phi_2^-$

\subsection{The transitions}\label{transitions}
The coordinates on $V_1$ and $V_2$ are chosen so that $[\sigma_1\rightarrow\sigma_2]$ connects points with $z>0$ (resp. $z<0$) in $V_1$ to points with $z>0$  (resp. $z<0$) in $V_2$. Points in $\Out(\sigma_1) \setminus W^u_{\lambda,\loc}(\sigma_1)$
are mapped into $\In(\sigma_2)$ along a flow-box around each of the connections $[\sigma_1\rightarrow\sigma_2]$. We will assume that the transition
$$\Psi_{1 \rightarrow 2}^\pm\,\colon \quad \Out(\sigma_1) \quad \rightarrow \quad \In(\sigma_2)$$
does not depend on $\lambda$ and is the Identity map, a choice compatible with hypothesis (P5) due to the fact that the nodes $\sigma_1$ and $\sigma_2$ have the same chirality. Denote by $\eta^+$ the first hit map
$$\eta^+=\Phi_{2}^+ \circ \Psi_{1 \rightarrow 2}^+ \circ \Phi_{1 }^+\colon \quad \In^+(\sigma_1) \quad \rightarrow \quad \Out^+(\sigma_2).$$
From \eqref{local_v} and \eqref{local_w} we infer that, in local coordinates, for $y>0$ we have
\begin{equation}\label{eqeta}
\eta^+(x,y)=\left(x-K \log y \,\,\,(mod\,2\pi), \,y^{\delta} \right) + o(y^\delta)
\end{equation}
with
\begin{equation}\label{delta e K}
\delta=\delta_1 \delta_2>1 \qquad \text{and} \qquad  K= \frac{C_1+E_2}{E_1 E_2} > 0.
\end{equation}
Analogously, we define the transition $\eta^- = \Phi_{2}^- \circ \Psi_{1 \rightarrow 2}^- \circ \Phi_{1 }^-$ in $\In^-(\sigma_1)$. The global first hit map is precisely
$$\eta \,\colon \In^+(\sigma_1)\cup \In^-(\sigma_1) \rightarrow \Out(\sigma_2)$$
given by $\eta =\eta^+$ in $\In^+(\sigma_1)$ and $\eta=\eta^-$ in $\In^-(\sigma_1)$. When $\lambda\neq 0$, we have another transition map, namely
$$\Psi_{2 \rightarrow 1}^\lambda \,\colon \quad \Out(\sigma_2)\rightarrow \In(\sigma_1)$$
which depends on the parameter $\lambda$. Due to the mentioned higher order terms and the assumption (P7), its expression is only approximately known, and this is why we will have to resort to a model (cf. Section \ref{modelo}). To build such a model, let us examine the images by $\Psi_{2 \rightarrow 1}^\lambda$ of the two-dimensional invariant manifolds of $\sigma_1$ and $\sigma_2$ when $\lambda\neq 0$.
\bigbreak

Denote by $P_1=(p_1,0)$ and $P_2=(p_2,0)$, with $0 \leqslant p_1 < p_2 \leqslant 2\pi$, the two points where the connections $[\sigma_2 \rightarrow \sigma_1]$ intersect $\Out(\sigma_2)$, and by $Q_1=(q_1,0)$ and $Q_2=(q_2,0)$, with $0 \leqslant q_2 < q_1 \leqslant 2\pi$, the two points where $[\sigma_2 \rightarrow \sigma_1]$ meets $\In(\sigma_1)$ (seen Figure~\ref{elipse}). Notice that, for each $j=1,2$, the points $P_j$ and $Q_j$ are on the same trajectory. As the manifolds $W^u_0(\sigma_2)$ and $W^s_0(\sigma_1)$ coincide when $\lambda=0$, we may assume that, for small $\lambda>0$, the manifold $W^s_{\lambda,\loc}(\sigma_1)$ (respectively $W^u_{\lambda,\loc}(\sigma_2)$) intersects the wall $\Out(\sigma_2)$ of the cylinder $V_2$ (respectively the wall $\In(\sigma_1)$ of the cylinder $V_1$) on a closed curve. This is precisely the expected unfolding from the coincidence when $\lambda=0$ of the manifolds $W^s_0(\sigma_1)$ and $W^u_0(\sigma_2)$, as established in \cite{Chillingworth}.

We will assume that, for $\lambda > 0$ the intersection $W^u_\lambda(\sigma_2) \cap \In(\sigma_1)$ is sinusoidal in shape. Nevertheless, any smooth $2\pi-$periodic function  would work  (cf. Section 7 of \cite{Champneys2009}). So, for small $\lambda>0$, these curves are such that $W^s_{\lambda,\loc}(\sigma_1)\cap \Out(\sigma_2)$ is the graph of a map $y=\xi_s(x,\lambda)$, with $\xi_s(p_j,\lambda)=0$ for $j=1,2$; $W^u_{\lambda,\loc}(\sigma_2)\cap \In(\sigma_1)$ is the graph of $y=\xi_u(x,\lambda)$, with $\xi_u(q_j,\lambda)=0$ for $j=1,2$; the maximum value of $\xi_u(x,\lambda)$ is attained at some point $(x_1(\lambda),y_1(\lambda))$ with $q_2 < x_1(\lambda) < q_1$; $\xi_s(x,0)\equiv 0 \equiv \xi_u(x,0)$; $\xi_s^\prime(p_1,\lambda)>0$, $\xi_s^\prime(p_2,\lambda)<0$, $\xi_u^\prime(q_1,\lambda)<0$ and $\xi_u^\prime(q_2,\lambda)>0$. Moreover, we have $\xi_u^{\prime\prime}(x)< 0$ in $]\, q_2,q_1\,[$. Check this technical details in Figure~\ref{elipse}.

\begin{figure}[h]
\begin{center}
\includegraphics[height=11cm]{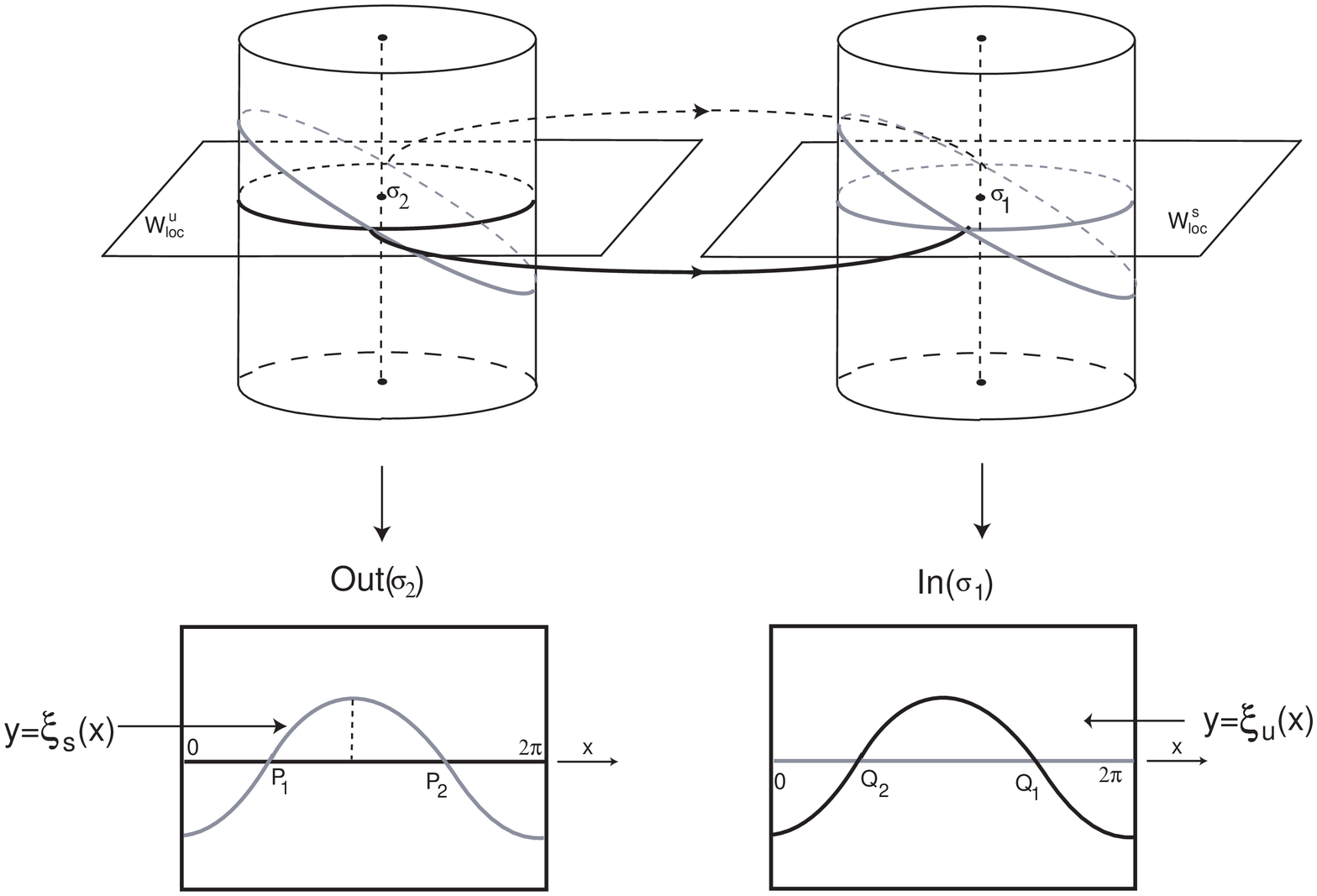}
\end{center}
\caption{\small For $\lambda \gtrsim 0$, $W^s_\lambda(\sigma_1)$ intersects the wall $\Out(\sigma_2)$ of the cylinder $V_2$ and $W^u_\lambda(\sigma_2)$ intersects the wall $\In(\sigma_1)$ of the cylinder $V_1$ on closed curves given, in local coordinates, by the graph of a $2\pi-$periodic function of sinusoidal type. In this work we assume that this sinusoidal curve is the graph of the sine map.}
\label{elipse}
\end{figure}

\section{A model for the transition}\label{modelo}
In order to deal with an workable expression of the first return map, we have chosen a simple transition map from $\Out(\sigma_2)$ to $\In(\sigma_1)$. This is the main specificity of the model we will consider.

\subsection{The model}
In what follows, we will assume that
\begin{eqnarray}\label{transition21}
\xi_u(x, \lambda) &=& \lambda \sin x 
\end{eqnarray}
and we get
\begin{eqnarray}
\Psi_{2 \rightarrow 1}^\lambda(x,y) &=& \left(x, \,y +\lambda \sin x\right) \nonumber\\
\smallskip
\eta^+(x,y) &=& \left(x-K \log y \,\,\,(mod\,2\pi), \,y^{\delta} \right) \nonumber.
\end{eqnarray}
Let
\begin{equation}\label{first return 1}
\mathcal{P}_\lambda \,:= \,\Psi_{2 \rightarrow 1}^\lambda \circ \eta = \Psi_{2 \rightarrow 1}^\lambda \circ \Phi_{2} \circ \Psi_{1 \rightarrow 2} \circ \Phi_{1} 
\end{equation}
be the first return map to $\In(\sigma_1)$. Its domain $\mathcal{D}$ is the set of initial conditions $(x,y) \in \In(\sigma_1)$ whose trajectory returns to $\In(\sigma_1)$. It follows from \eqref{transition21} that, when $y>0$, the first return map at $(x,y)$, say $\mathcal{P}_\lambda(x,y)= \left(\mathcal{P}_{\lambda,1}(x,y), \,\mathcal{P}_{\lambda,2}(x,y)\right)$, is given by
\begin{equation}\label{first return 2}
\mathcal{P}_\lambda(x,y) = \left(x-K\log y \,\,\, (mod\,2\pi), \,y^\delta + \lambda \sin(x-K\log y)\right).
\end{equation}
Besides,
$$D \Psi_{2 \rightarrow 1}^\lambda(x,y)=\left(\begin{array}{cc} 1 & 0\\
\frac{\partial \xi_u}{\partial x} & 1 \end{array}\right) = \left(\begin{array}{cc} 1 & 0\\
\lambda \cos(x) & 1 \end{array}\right).$$
Therefore, as
$$D\eta(x,y)=\left(\begin{array}{cc} 1&\dpt\frac{-K}{y}\\ \\ 0&\delta y^{\delta-1}\end{array}\right)
$$
we get
$$\det D \mathcal{P}_\lambda(x,y)= \det D \Psi_{2 \rightarrow 1}^\lambda\left(\eta(x,y)\right) \times \det D\eta(x,y)=\delta y^{\delta-1}.$$
Thus, for sufficiently small $y>0$, the first return map $\mathcal{P}_\lambda$ is contracting. After appropriate adaptations, we obtain similar expressions for $y<0$.

\subsection{First return time}\label{sse:first-return-time}
One of the advantages of considering the previous model is that we will be able to explicitly compute the time of flight around the heteroclinic cycle of a periodic solution that turns once around $\Sigma_\lambda$. A trajectory whose initial condition is $(1,\theta,z) \in \In^+(\sigma_1)$ arrives at $\Out^+(\sigma_1)$ after a period of time equal to
\begin{equation}\label{time1}
-\,\frac{\log z}{E_1}.
\end{equation}
Consequently, given $(x,y)\in \mathcal{D} \subset \In^+(\sigma_1)$, if $t_1(x,y) \in \RR^+$ denotes the time needed for its first return to $\In(\sigma_1)$, then
\begin{equation}\label{Time1}
t_1(x,y) = \left|\frac{1}{E_1}\log y \right| + \left|\frac{1}{E_2}\frac{C_1}{E_1} \log y\right| = \left| \frac{E_2+C_1}{E_1E_2}\log y \right|= - K \log y.
\end{equation}
Analogous formulas for $(x,y)\in  \In^-(\sigma_1)$.

\section{Proof of Theorem~\ref{thm:A}}\label{Proof of Theorem A}

Recall that we have fixed a tubular neighborhood $\mathcal{T}$ of $\Sigma_\lambda^\star$ where we may find a countable union of topological horseshoes $\Big(\Lambda_{\lambda, m}\Big)_{m\, \in\,  \NN}$, denoted by $\Lambda_\lambda$, whose saturation $\widetilde{\Lambda_\lambda}$ by the flow does not leave $\mathcal{T}$ (cf. \cite{ACL05, Wiggins}).  For each $m\in \NN$, the set $\Lambda_{\lambda, m}$ is contained in a rectangle  $[-\tau, \tau]\times[a_m, b_m]$ inside $\In(\sigma_1)$, for sequences $(a_m)_{m\, \in\, \NN}$  and  $(b_m)_{m\, \in\, \NN}$  explicitly given in \cite{Rodrigues3}. 
Having fixed $\lambda>0$, in this section we focus our attention on the dynamics inside two rectangles centered at the two connections $\beta_1$ and $\beta_2$, whose existence is guaranteed by (P7). 
\subsection{Linking $\widetilde{\Lambda_\lambda}$ with $W^s_\lambda(\sigma_1)$ and $W^u_\lambda(\sigma_2)$}
\label{ss:heteroclinic class}

We start showing that $\widetilde{\Lambda_{\lambda}}$ is contained in the closure of the intersections of the two-dimensional manifolds $W^u_\lambda(\sigma_2)$ and $W^s_\lambda(\sigma_1)$, the so called heteroclinic class of $\sigma_1$ and $\sigma_2$. Let $p \in{\Lambda_\lambda}$.

From the construction of $\Lambda_\lambda$ (cf.  \cite[\S 2.3]{Wiggins}), there exist a sequence $(s_n)_{n \in \ZZ}$ of integers,
a sequence of horizontal strips $(H_{s_0 s_{1}\ldots s_{n}})_{n \,\in \,\NN}$ and a sequence of vertical strips $(V_{s_0 s_{1}\ldots s_{n}})_{n \,\in \,\NN}$ such that the diameter of the sequence of intersections $\left(V_{s_0 s_{-1}\ldots s_{-n}}\cap H_{s_0 s_{1}\ldots s_{n}}\right)_{n \,\in \,\NN}$ goes to zero and, for every $n \in \NN_0$,
$$p \in V_{s_0 s_{-1}\ldots s_{-n}} \,\cap \,H_{s_0 s_{1}\ldots s_{n}}.$$
Therefore, given $\varepsilon>0$,
there is $\ell \in \NN$ large enough such that the diameter of the intersection $V_{s_0 s_{-1}\ldots s_{-\ell}}\cap H_{s_0 s_{1}\ldots s_{\ell}}$ is smaller than $\varepsilon$. Moreover, $\ell$ may be chosen sufficiently large so that, in $V_{s_0 s_{-1}\ldots s_{-\ell}}$, the invariant sets $W^u_{\lambda,\loc}(\sigma_2)$ and $W^s_{\lambda,\loc}(\sigma_1)$ are vertical and horizontal fibers, respectively, which intersect transversely. Thus, $p \in \overline{W^u_\lambda(\sigma_2)\cap W^s_\lambda(\sigma_1)}$. Therefore  $\widetilde{\Lambda_{\lambda}} \subset \overline{W^u_\lambda(\sigma_2)\cap W^s_\lambda(\sigma_1)}$.

\subsection{Linking $\sigma_1$ and $\sigma_2$ with $\widetilde{\Lambda_\lambda}$}\label{Proof of Theorem A2}

Firstly, let us verify the inclusion $\overline{W^u_\lambda(\widetilde{\Lambda_\lambda})}\subset \overline{W^u_\lambda(\sigma_2)}$. Let $P$ be an accumulation point of ${W^u_\lambda(\widetilde{\Lambda_\lambda})}$. Again, by the construction of $\Lambda_{\lambda}$,
there is a sequence $(V_n)_n$ of vertical rectangles whose vertical boundaries are contained in $W^u_\lambda(\sigma_2)$, their widths approaches $0$ as $n$ goes to $+ \infty$ and $\Lambda_{\lambda} \subset V_n$ for every $n \in \NN$. By the continuity, in the adequate topologies, of the return map $\mathcal{P}_\lambda$ and of the compact parts of the invariant manifolds involved, for $n$ sufficiently large the vertical boundary of $V_n$ is arbitrarily close to $W^u_\lambda(\widetilde{\Lambda_\lambda})$. Thus, $P$ is an accumulation point of $W^u_\lambda(\sigma_2)$.


\subsection{The fixed points of $\mathcal{P}_\lambda$}\label{sse:fixedpoints}
Recall from \eqref{first return 2} that the first return map in $\mathcal{D} \cap \In^+(\sigma_1)$ of the model is given by
$$\mathcal{P}_\lambda(x,y)=\left( x-K \log y \pmod{2\pi},\,\, y^\delta+\lambda\sin(x-K\log(y))\right)$$
with $\delta > 1$, $K>0$ and $\mathcal{D}$ as in \eqref{first return 1}. So, its fixed points in $\mathcal{D}\cap \In^+(\sigma_1)$, say $p_\ell=(x_\ell, y_\ell)\in \In^+(\sigma_1)$ with $\ell \in \NN$, are solutions of the system of equations
\begin{equation*}
\left\{
\begin{array}{l}
x-K \log (y)=x+2\ell \pi\\
\smallskip \\
y^\delta +\lambda \sin(x-K\log(y))=y.
\end{array}
\right.
\end{equation*}
Thus,
$$y_\ell =  \exp \left(\frac{-2\ell \pi}{K}\right) \qquad \text{and} \qquad \sin(x_\ell)=\frac{y_\ell-y_\ell^\delta}{\lambda}=\frac{\exp \left(\frac{-2\ell \pi}{K}\right)-\exp \left(\frac{-2\ell\delta \pi}{K}\right)}{\lambda},$$
provided that $\left|\frac{y_\ell\,-\,y_\ell^\delta}{\lambda}\right|\leqslant 1$ or, equivalently, while it is true that
\begin{equation}\label{lambda maximo}
\lambda \geqslant \exp \left(\frac{-2\ell \pi}{K}\right) - \exp \left(\frac{-2\ell \delta \pi}{K}\right).
\end{equation}
Consequently, as soon as $\lambda$ becomes smaller that $\exp \left(\frac{-2\ell \pi}{K}\right) - \exp \left(\frac{-2\ell \delta \pi}{K}\right)$, the fixed point $p_\ell$ disappears. As $\Lambda_{\lambda,\ell}$ is contained in the homoclinic class $\overline{W^s_\lambda(p_{\lambda,\ell})\cap W^u_\lambda(p_{\lambda,\ell})}$ of the fixed point $p_{\lambda,\ell}$, when the fixed point $p_{\lambda,\ell}$ disappears this horseshoe no longer exists. 

\begin{lemma}[\cite{ACL05}]\label{prop:non-unif-hyp-Lambda}
Given $\lambda>0$, the derivative $D\mathcal{P}_\lambda$ at  $p_\ell=(x_\ell,y_\ell) \in \mathcal{D}\cap \In^+(\sigma_1)$ has real eigenvalues $\mu_s(p_\ell)$ and $\mu_u(p_\ell)$ and eigenvectors $V_s(p_\ell)$ and $V_u(p_\ell)$ such that:
\begin{enumerate}
\item $\lim_{\ell \rightarrow +\infty} \mu_s(p_\ell)=0^-\quad $ and $\quad \lim_{\ell \rightarrow +\infty} \mu_u(p_\ell)=-\infty$.
\medskip
\item $\lim_{\ell \rightarrow +\infty} \frac{V_s(p_\ell)}{||V_s(p_\ell)||}=\left\langle 1,0 \right\rangle \quad$ and $\quad \lim_{\ell \rightarrow +\infty} \frac{V_u(p_\ell)}{||V_u(p_\ell)||}=\left\langle 1,\lambda  \right\rangle$.
\end{enumerate}
\end{lemma}

The argument to prove this lemma is essentially the same as the proof of Theorem 6 of \cite{ACL05}. We include the explicit expressions for the trace and the determinant for future use.  Having fixed $\lambda>0$, the derivative of $\mathcal{P}_\lambda$ in $\mathcal{D}\cap \In^+(\sigma_1)$ is given by
\begin{equation}\label{eq:derivada}
D\mathcal{P}_\lambda(x,y)=\left(\begin{array}{cc}
1 & -\frac{K}{y}\\
\lambda \cos(x-K\log(y)) \qquad  \qquad   & \delta y^{\delta-1} -\frac{\lambda K}{y}\cos(x-K\log(y))\\
\end{array}
\right).
\end{equation}
Consequently, the determinant of $D\mathcal{P}_\lambda(x,y)$ is positive and equal to $\delta y^{\delta-1}$ and its trace is given by $1+\delta y^{\delta-1} -\frac{\lambda K}{y}\cos(x-K\log(y)).$ At the fixed point $p_\ell=(x_\ell, y_\ell)$ we have $\det D\mathcal{P}_\lambda(p_\ell)> 0$ and
$\lim_{\ell \rightarrow +\infty} \det D\mathcal{P}_\lambda(p_\ell) = 0.$
In addition, the trace of $D\mathcal{P}_\lambda(p_\ell)$ may be rewritten as
\begin{eqnarray*}
 \text{Trace}\, D\mathcal{P}_\lambda(p_\ell)&=& 1 + \delta \exp \left(\frac{-2\pi(\delta-1)\ell}{K}\right)-\lambda K \exp \left(\frac{2\pi\ell}{K}\right)\cos(x_\ell)\\
 &=& 1 + \delta \exp \left(\frac{-2\pi(\delta-1)\ell}{K}\right)-\lambda K \exp \left(\frac{2\pi\ell}{K}\right)\sqrt{1-\sin^2(x_\ell)}=\\
&=& 1 + \delta \exp \left(\frac{-2\pi(\delta-1)\ell}{K}\right)- \\
&-& K\exp \left(\frac{2\pi\ell}{K}\right)\sqrt{\lambda^2-\left[\exp{\left(\frac{-2\ell \pi}{K}\right)} -\exp{\left(\frac{-2\ell \delta \pi}{K}\right)}\right]^2}.
\end{eqnarray*}
Thus $\lim_{\ell \rightarrow +\infty} \text{Trace}\,D\mathcal{P}_\lambda(p_\ell) = -\infty$ if $\lambda > \exp \left(\frac{-2\ell \pi}{K}\right) - \exp \left(\frac{-2\ell \delta \pi}{K}\right).$
\medbreak
Denote by $\det$ and $\text{trace}$ the determinant and the trace of $D\mathcal{P}_\lambda(p_\ell)$, respectively. Given $\lambda>0$ and $\ell \in \NN$ such that
$$\lambda^2-\left[\exp{\left(\frac{-2\ell \pi}{K}\right)} -\exp{\left(\frac{-2\ell \delta \pi}{K}\right)}\right]^2 > 0 \qquad \quad \text{and} \qquad \quad \text{trace}^2 - 4\det \geqslant 0,$$
the eigenvalues of $D\mathcal{P}_\lambda(p_\ell)$ satisfy the inequalities
$$\mu_u=\frac{\text{trace}-\sqrt{\text{trace}^2 - 4\det}}{2} \leqslant \frac{\text{trace}}{2} \quad \quad \text{and} \quad \quad \mu_s = \frac{\text{trace}+\sqrt{\text{trace}^2 - 4\det}}{2}<0.$$
Since $\mu_s \times \mu_u = \det$ and $\mu_s + \mu_u = \text{trace}$, we conclude that $\lim_{\ell \rightarrow +\infty} \mu_s(p_\ell) = 0^-$ and $\lim_{\ell \rightarrow +\infty} \mu_u(p_\ell)= -\infty.$

\section{The evolution of the eigenvalues $\mu_s(p_\ell)$ and $\mu_u(p_\ell)$}
\label{s: eigenvalues}
The values of $\mu_s(p_\ell)$ and $\mu_u(p_\ell)$ change with both $\ell$ and $\lambda$ as depicted in Figures \ref{saddle-node1} and \ref{Lamerey}. The images on the left were obtained using Maple and the expressions of the {trace} and {determinant} of the derivative of the first return map; the image at the top right is a scheme of what we conjecture to be the global bifurcations diagram.

 The analysis of these eigenvalues shows that, for each $\ell \in \NN$, there are parameters $a_\ell < b_\ell < c_\ell < d_\ell$ in $[0,1]$ at which relevant dynamical changes occur. For $\lambda < a_\ell$, there is no fixed point $p_\ell$. At $\lambda=a_\ell$, the fixed point $p_\ell$ is created through a saddle-node bifurcation yielding a sink and a saddle. In the interval $]\,b_\ell, c_\ell\,[$, the eigenvalues of the sink $p_\ell$ are complex, so it is a stable focus. At $\lambda=d_\ell$, the sink $p_\ell$ experiences another bifurcation, becoming a saddle. As expected (cf. \cite{YA}), cascades of saddle-nodes and period-doubling bifurcations happen after the parameter $d_\ell$; and afterwards the stable and unstable manifolds of the fixed point $p_\ell$ unfold a tangency, thus creating the horseshoe $\Lambda_{\lambda,\ell}$ to which $p_\ell$ belongs.

\begin{figure}[h]
\begin{center}
\includegraphics[height=13cm]{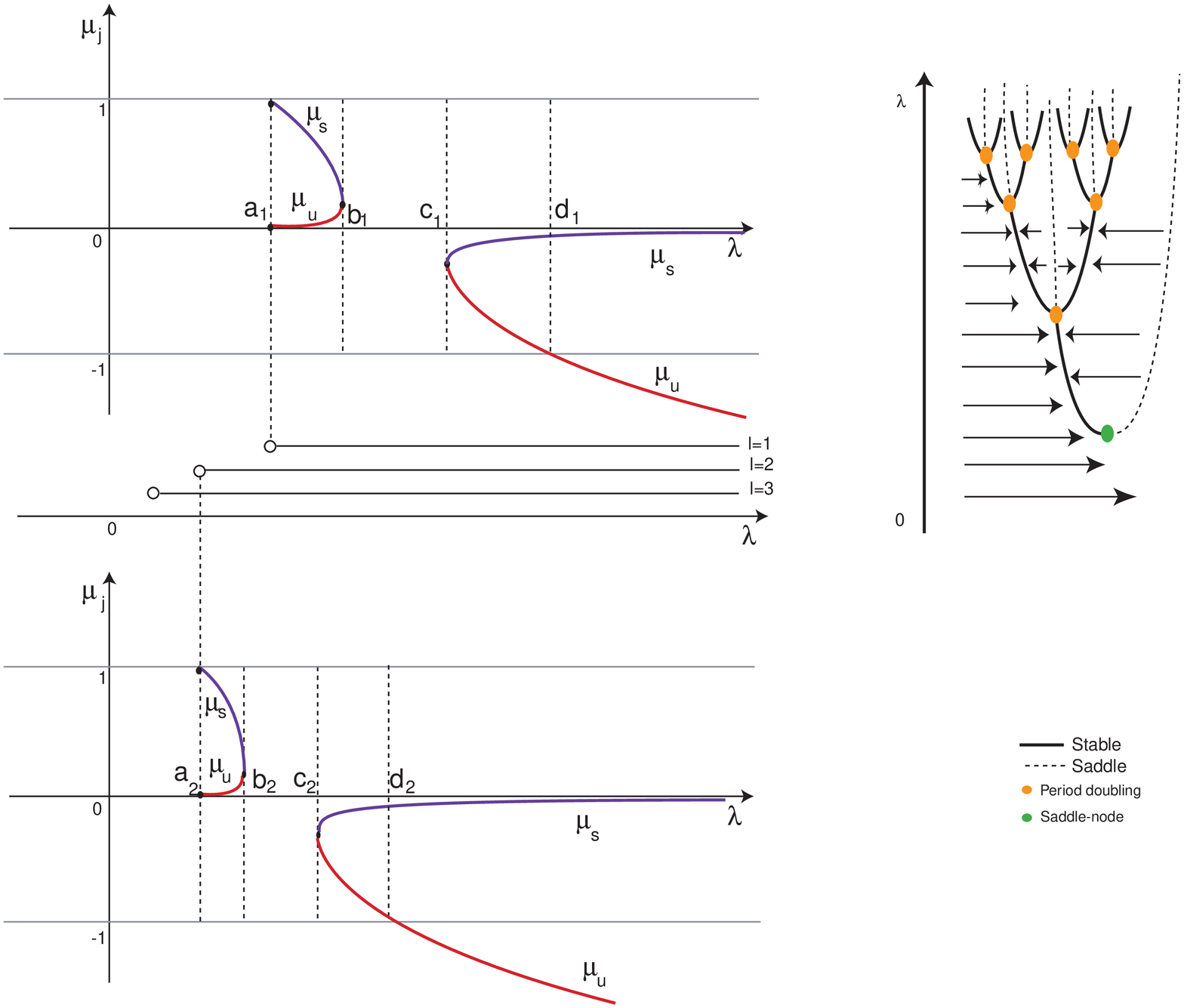}
\end{center}
\caption{\small Left: Graphs of the eigenvalue-maps as functions of $\lambda$ for $\ell = 1,2$, $\delta=4$ and $K=0.2$. Right: A component of the global bifurcation diagram of $(f_\lambda)_\lambda$. }
\label{saddle-node1}
\end{figure}

\begin{figure}[h]
\begin{center}
\includegraphics[height=6.0cm]{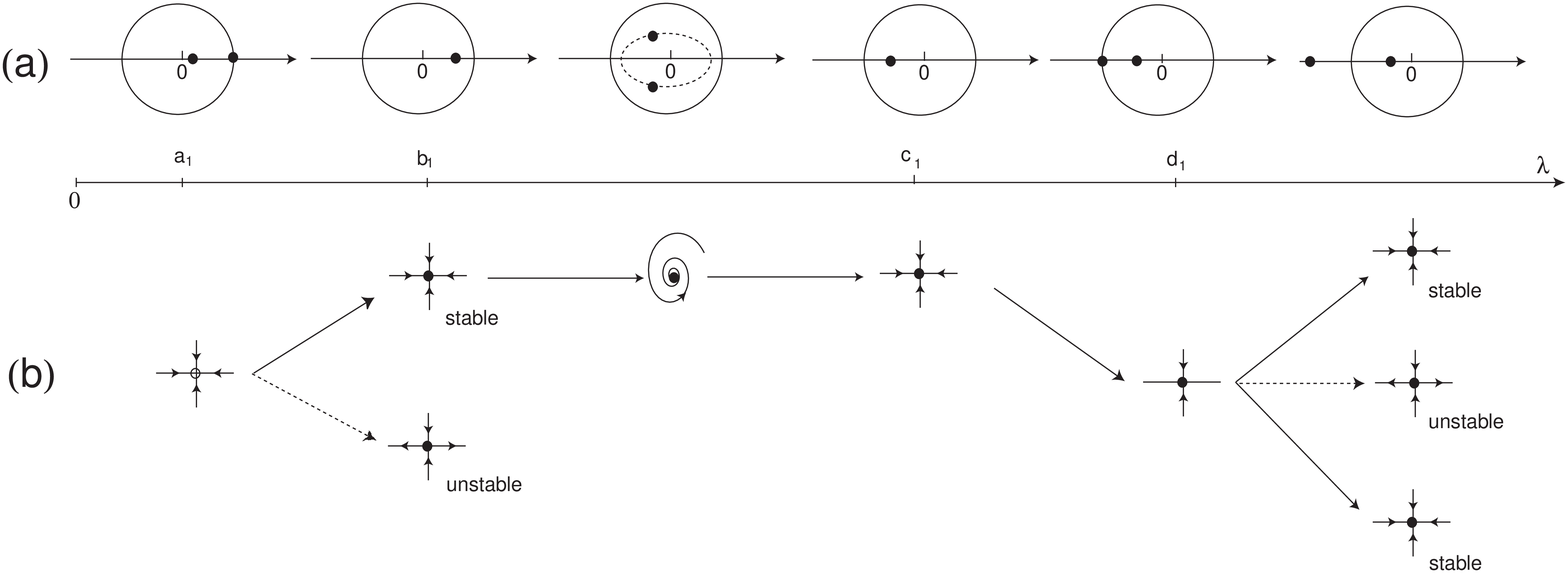}
\end{center}
\caption{\small Lamerey diagram for the parameters of Figure~\ref{saddle-node1} and $\ell=1$. (a) Evolution of the eigenvalues within the unit circle. (b) Corresponding changes in the nature of the fixed point $p_1$. }
\label{Lamerey}
\end{figure}

The stable and unstable eigenspaces of $p_\ell$ are the solutions of the linear equations
\begin{equation*}
D\mathcal{P}_\lambda(p_\ell)
\left(
\begin{array}{l}
v_1^s\\
v_2^s
\end{array}
\right)=
\mu_s(p_\ell)\left(
\begin{array}{l}
v_1^s\\
v_2^s
\end{array}
\right)
\qquad
\text{and}
\qquad
D\mathcal{P}_\lambda(p_\ell)
\left(
\begin{array}{l}
v_1^u\\
v_2^u
\end{array}
\right)=
\mu_u(p_\ell) \left(
\begin{array}{l}
v_1^u\\
v_2^u
\end{array}
\right).
\end{equation*}
Thus, concerning the eigenspace of $\mu_s(p_\ell)$, if we choose $v_1^s=1$ then we get
$$\lim_{\ell \rightarrow +\infty} \,v_2^s(p_\ell)= \lim_{\ell \rightarrow +\infty}  \frac{1-\mu_s(p_\ell)}{K}\,y_\ell =0,$$
which means that the one-dimensional stable direction tends to be horizontal as $y_\ell$ approaches $0$. Similarly, if $v_1^u=1$ then
\begin{eqnarray*}
v_2^u(p_\ell)&=& \frac{1-\mu_u(p_\ell)}{K}\,y_\ell = \frac{1}{K}\,y_\ell - \frac{\text{trace}-\mu_s(p_\ell)}{K}\,y_\ell = \left(\frac{1}{K}+ \frac{\mu_s(p_\ell)}{K}\right)\,y_\ell - \frac{\text{trace}}{K}\,y_\ell =\\
&=& \left(\frac{1}{K}+ \frac{\mu_s(p_\ell)}{K}\right)\,y_\ell- \frac{1}{K} \left[1+\delta y_\ell^{\delta-1} - \frac{\lambda K}{y_\ell}\cos(x_\ell-K\log(y_\ell))\right]\,y_\ell=\\
&=&\left(\frac{1}{K}+ \frac{\mu_s(p_\ell)}{K}\right)\,y_\ell - \left(\frac{1+\delta y_\ell^{\delta-1}}{K}\right)\,y_\ell +\lambda \,\sqrt{1-\left(\frac{y_\ell-y_\ell^\delta}{\lambda}\right)^2}.
\end{eqnarray*}
Consequently,
$\lim_{\ell \rightarrow +\infty}\, v_2^u(p_\ell)=\lambda,$
meaning that the unstable one-dimensional space tends to the direction spanned by $(1,\lambda)$, which is precisely the slope of the tangent to the graph of $\xi_u$ at the primary connection.

\section{Lyapunov exponents of $\mathcal{P}_\lambda$ at $p_\ell$}\label{ss:LE}

The Lyapunov exponents along the orbit of a fixed point of $\mathcal{P}_\lambda$ depend on the time needed for the trajectory to return to the cross section, whose main components are the sojourns spent near $\sigma_1$ and $\sigma_2$. Using the computations of Section \ref{s: eigenvalues}, we may now estimate them.
\medbreak

The Lyapunov exponents at a fixed point $p_\ell \in \mathcal{D}\cap \In^+(\sigma_1)$ of $\mathcal{P}_\lambda$ are approximately given by
$$\chi_s (p_\ell) \simeq \frac{\log |\mu_s(p_\ell)|}{t_1(p_\ell)}  \quad \text{and} \quad \chi_u(p_\ell) \simeq \frac{\log |\mu_u(p_\ell)|}{t_1(p_\ell)}$$
where $t_1(p_\ell)$ is the time of first return of $(p_\ell)$ to the cross section by the iteration of $\mathcal{P}_\lambda$. As $t_1(p_\ell) = -K \log y_\ell$ (see Subsection~\ref{sse:first-return-time}), $\delta > 1$ and $K > 0$ (cf. \eqref{delta e K}), we will use the calculations of the previous subsection to estimate these exponents as follows.
Firstly, using \eqref{Time1}, we get
\begin{eqnarray*}
\chi_s(p_\ell) + \chi_u(p_\ell) &\simeq& \frac{\log |\mu_s(p_\ell) \times \mu_u(p_\ell)|}{-K\log y_\ell} = \frac{\log \,(\delta\,y_\ell^{\delta-1})}{-K\log y_\ell}
= \frac{\log \delta}{-K\log y_\ell} + \frac{\delta-1}{-K}
\end{eqnarray*}
and so
\begin{equation}\label{conta1}
\lim_{\ell \rightarrow +\infty} \, \left(\chi_s(p_\ell) + \chi_u(p_\ell)\right) = - \frac{\delta-1}{K}.
\end{equation}
On the other hand, we have
\begin{eqnarray*}
\frac{\log |\mu_s(p_\ell) + \mu_u(p_\ell)|}{-K\log y_\ell} &=& \frac{\log \,\left(|\mu_u(p_\ell)|\left|1 + \frac{\mu_s(p_\ell)}{\mu_u(p_\ell)}\right|\right)}{-K\log y_\ell}=
 \frac{\log |\mu_u(p_\ell)|}{-K \log y_\ell} + \frac{\log \left|1 + \frac{\mu_s(p_\ell)}{\mu_u(p_\ell)}\right|}{-K\log y_\ell} \\
 &\simeq& \chi_u(p_\ell) + \frac{\log \left|1 + \frac{\mu_s(p_\ell)}{\mu_u(p_\ell)}\right|}{-K\log y_\ell}
\end{eqnarray*}
where
$$\lim_{\ell \rightarrow +\infty} \, \frac{\log \left|1 \,+\, \frac{\mu_s(p_\ell)}{\mu_u(p_\ell)}\right|}{-K\log y_\ell} = 0.$$
Moreover,
\begin{eqnarray*}
&&\lim_{\ell \rightarrow +\infty} \, \frac{\log |\mu_s(p_\ell) + \mu_u(p_\ell)|}{-K\log y_\ell} = \lim_{\ell \rightarrow +\infty} \, \frac{\log |\text{trace}\, D\mathcal{P}_\lambda(p_\ell)|}{-K\log y_\ell} = \\
&=&  \lim_{\ell \rightarrow +\infty} \, \frac{\log \,\left(-1-\delta y_\ell^{\delta-1} + \frac{\lambda K}{y_\ell}\cos(x_\ell-K\log y_\ell)\right)}{-K\log y_\ell}=\\
&=& \lim_{\ell \rightarrow +\infty} \, \frac{\log \,\left(-1-\delta y_\ell^{\delta-1} + \frac{\lambda K}{y_\ell}\sqrt{1-\left(\frac{y_\ell-y_\ell^\delta}{\lambda}\right)^2}\,\right)}{-K\log y_\ell}=\\
&=& \lim_{\ell \rightarrow +\infty} \, \frac{\log \left(\frac{-y_\ell-\delta y_\ell^{\delta} +  K\sqrt{\lambda^2-\left(y_\ell-y_\ell^\delta\right)^2}}{y_\ell}\right)}{-K\log y_\ell}=\\
&=& \lim_{\ell \rightarrow +\infty} \, \frac{-\log y_\ell}{-K\log y_\ell} + \lim_{\ell \rightarrow +\infty} \, \frac{\log \, \left(-y_\ell-\delta y_\ell^{\delta} +  K\sqrt{\lambda^2-(y_\ell-y_\ell^\delta)^2}\,\right)}{-K\log y_\ell}
=  \frac{1}{K}.
\end{eqnarray*}
Thus,
$$\lim_{\ell \rightarrow +\infty} \chi_u(p_\ell) =  \frac{1}{K}>0$$
and so, by \eqref{conta1},
$$\lim_{\ell \rightarrow +\infty} \chi_s(p_\ell) = - \frac{\delta}{K}<0.$$

\section{Proof of Theorem \ref{thm:C}}\label{Proof of Theorem C}

Let $V^0$ be as in Section~\ref{overview} and $\mathcal{B}_\lambda = \{P \in V^0: P\, \text{ is chain-accessible from } \,\sigma_2\}.$ We first observe that $\mathcal{B}_\lambda$ coincides with $\mathcal{C}_\lambda=\{P \in V^0: P\, \text{ is chain-accessible from } \,\sigma_1\}.$ This is a general property of any heteroclinic network linking two equilibria and a straightforward consequence of the following two facts. The equilibrium $\sigma_2$ is chain-accessible from $\sigma_1$ because, given $\varepsilon>0$ and $\tau>0$, any point of $[\sigma_1 \rightarrow \sigma_2] \setminus \{\sigma_1\}$ at a distance from $\sigma_1$ smaller than $\varepsilon$ is an $(\varepsilon, \tau)$--trajectory connecting $\sigma_1$ with $\sigma_2$. Thus, $\mathcal{B}_\lambda \subset \mathcal{C}_\lambda$. Conversely, $\sigma_1$ is chain-accessible from $\sigma_2$ due to the fact that, given $\varepsilon>0$ and $\tau>0$, any point $P_1 \neq \sigma_2$ of a $0$--pulse at a distance from $\sigma_2$ smaller than $\varepsilon$ is an $(\varepsilon, \tau)$--trajectory connecting $\sigma_2$ with $\sigma_1$. So $\mathcal{C}_\lambda \subset \mathcal{B}_\lambda$.

We proceed showing that $\mathcal{B}_\lambda$ is forward invariant by the flow, closed and chain-stable.

\subsection{Forward invariant} Given $\lambda \geqslant 0$, $P \in \mathcal{B}_\lambda$ and $t \geqslant 0$, if $\{P_1,\cdots,P_k\}$ is an $(\varepsilon,\tau)$--trajectory connecting $\sigma_2$ to $P$, then $\{P_1,\cdots,P_k, P\}$ is an $(\varepsilon,\min\,\{\tau,t\})$--trajectory connecting $\sigma_2$ to $\varphi_\lambda(t,P)$. Thus, $\varphi_\lambda(t,\mathcal{B}_\lambda) \subset \mathcal{B}_\lambda$ for every $t \geqslant 0$; that is, $\mathcal{B}_\lambda$ is forward invariant.

\subsection{Closed} We now verify that accessibility from $\sigma_2$ (and more generally from any point) is a closed property. Let $Q$ be a point in $\overline{\mathcal{B}_\lambda}$ and $\varepsilon>0$ arbitrarily small. Take the open ball $B_{\frac{\varepsilon}{2}}(Q)$ centered at $Q$ with radius $\frac{\varepsilon}{2}$ and $P \in \mathcal{B}_\lambda \cap B_{\frac{\varepsilon}{2}}(Q)$. By definition of $\mathcal{B}_\lambda$, there are $\tau> 0$ and an $(\frac{\varepsilon}{2},\tau)$--trajectory $\{P_1,\cdots,P_k\}$ connecting $\sigma_2$ to $P$. Therefore, the distance between $P_k$ and $Q$ is smaller than $\varepsilon$, and so $\{P_1,\cdots,P_k\}$ is an $(\frac{\varepsilon}{2},\tau)$--trajectory connecting $\sigma_2$ to $Q$ as well. Thus $Q \in \mathcal{B}_\lambda$ and $\mathcal{B}_\lambda$ is closed. Notice also that, if $P \in \mathcal{B}_\lambda$ and $\{P_1,\cdots,P_k\}$ is an  $(\varepsilon,\tau)$--trajectory connecting $\sigma_2$ to $P$, then any $P_j$ is, by definition, chain-accessible from $\sigma_2$. Thus, $P_j \in \mathcal{B}_\lambda$ for every $1 \leqslant j \leqslant k$.

\subsection{Chain-stable} Let $(\varepsilon_j)_{j \in \NN}$ be a decreasing sequence of positive real numbers converging to $0$
and define $\mathcal{U}_j$ as the set of points $P \in V^0$ that are chain-accessible from $\sigma_2$ by an $(\varepsilon_j,\tau)$--trajectory for some $\tau > 0$. Notice that each $\mathcal{U}_j$ is open. Indeed, given $P \in \mathcal{U}_j$, let $\{P_1,\cdots,P_k\}$ an $(\varepsilon_j,\tau)$--trajectory connecting $\sigma_2$ to $P$. Then the open ball $B_\rho(P)$ centered at $P$ with radius
$\rho = \frac{\varepsilon_j-\text{distance}\,(P,P_k)}{2}>0$ is contained in $\mathcal{U}_j$ since, for every $Q \in B_\rho(P)$, we have
\begin{eqnarray*}
\text{distance}\,(Q, P_k) &\leqslant& \text{distance}\,(Q, P) + \text{distance}\,(P,P_k) <  \rho + \text{distance}\,(P,P_k)\\
&<& \frac{\varepsilon_j+ \varepsilon_j}{2}= \varepsilon_j.
\end{eqnarray*}
Therefore, $\{P_1,\cdots,P_k\}$ is an $(\varepsilon_j,\tau)$--trajectory also connecting $\sigma_2$ to $Q$.
Moreover, every $(\varepsilon_j,\tau)$--trajectory connecting two points of $\mathcal{B}_\lambda$ is contained in $\mathcal{U}_j$. Hence $\mathcal{B}_\lambda=\bigcap_j\,\mathcal{U}_j$, and $(\mathcal{U}_j)_{j \in \NN}$ is a family of nested neighborhoods of $\mathcal{B}_\lambda$. Now, given an open neighborhood $\mathcal{V}$ of $\mathcal{B}_\lambda$, there exists $j \in \NN$ such that $\mathcal{U}_j \subset \mathcal{V}$. So every $(\varepsilon_j,\tau)$--trajectory connecting two points of $\mathcal{B}_\lambda$ is contained in $\mathcal{V}$. Thus $\mathcal{B}_\lambda$ is chain-stable.

\subsection{Relevant subsets of $\mathcal{B}_\lambda$} We are left to show that $\overline{W^u_\lambda (\sigma_2)\,\cup\,W^s_\lambda (\sigma_1)} \subset \mathcal{A}_\lambda \cap  \mathcal{B}_\lambda.$
We start recalling that both manifolds $W^s_\lambda(\sigma_1)$ and $W^u_\lambda(\sigma_2)$ are contained in  $\mathcal{A}_\lambda$ (see \eqref{union_final}). It is also immediate to verify that the manifold $W^u_\lambda(\sigma_2)$, and so its closure, is contained in $\mathcal{B}_\lambda$. Indeed, take $P \in W^u_\lambda(\sigma_2)$, $\varepsilon >0$ and $t > 0$ such that $\varphi_\lambda(-t,P)$ is at a distance smaller than $\varepsilon$ from $\sigma_2$. Then the $(\varepsilon, t)$--trajectory $\{\varphi_\lambda(-t,P)\}$ connects $\sigma_2$ to $P$; thus, $P$ belongs to $\mathcal{B}_\lambda$.  Observe now that $W^s(\sigma_1) \subset \mathcal{C}_\lambda=\mathcal{B}_\lambda$, hence $\overline{W^s_\lambda(\sigma_1)} \subset \mathcal{B}_\lambda$.

\subsection{Chain-recurrence}
For $\lambda=0$ the attracting set $\mathcal{A}_0$ is a chain-recurrent class and coincides with $\mathcal{B}_0$. When $\lambda > 0$, by \cite[Lemma 3.2 and Remark 4.6]{HSZhao01} or \cite[Theorem A]{BF85}, the set $\mathcal{B}_\lambda$ is a chain-recurrent class unless it contains proper attractors.
We observe that, if $\mathcal{A}_\lambda$ and $\mathcal{B}_\lambda$ are chain-recurrent classes, one has $\mathcal{A}_\lambda=\mathcal{B}_\lambda$. Indeed, on the one hand, $\mathcal{B}_\lambda \subset \mathcal{A}_\lambda$ since we are assuming that both $\mathcal{A}_\lambda$ and $\mathcal{B}_\lambda$ are chain-recurrent classes, $\mathcal{A}_\lambda$ is an attracting set and $\mathcal{A}_\lambda \cap \mathcal{B}_\lambda \neq \emptyset$. On the other hand, as $\sigma_2$ belongs to $\mathcal{A}_\lambda$ and this set is a chain-recurrent class, any of its points is chain-accessible from $\sigma_2$; therefore $\mathcal{A}_\lambda \subset \mathcal{B}_\lambda$.

However, we do not know whether $\mathcal{A}_\lambda = \mathcal{B}_\lambda$ when $\lambda > 0$, and it seems unlikely to find $\lambda$ whose $\mathcal{A}_\lambda$ does not contain proper attractors (though we do know that there are many parameters for which the flow exhibits proper attractors (cf. items (4)--(6) of Section~\ref{overview}). Yet, it is easy to conclude that the subset of $\mathcal{A}_\lambda$ whose elements are accumulation points of trajectories starting at initial conditions in $\mathcal{B}_\lambda$ is contained in $\mathcal{B}_\lambda$. Indeed, consider $\varepsilon > 0$ and $P \in \mathcal{A}_\lambda$ for which there exists $Q \in \mathcal{B}_\lambda$ whose $\omega$-limit contains $P$. Take $\tau > 0$ and an $(\varepsilon, \tau)$-trajectory $\{P_1, \cdots, P_k\}$ connecting $\sigma_2$ to $Q$. As $P$ is in the $\omega$-limit of $Q$, there is $t > \tau$ such that the distance between $P$ and $\varphi_\lambda(t,Q)$ is smaller than $\varepsilon$. Therefore, the set $\{P_1, \cdots,P_k, Q\}$ is an $(\varepsilon, \tau)$-trajectory connecting $\sigma_2$ to $P$. So $P$ belongs to $\mathcal{B}_\lambda$.

\subsection{Final remark}
For $\lambda=0$, the manifold $W^s_0(\sigma_1) \equiv W^u_0(\sigma_2)$ is chain-accessible from any point of $V^0$. In fact, as $\mathcal{A}_0= [\sigma_1 \rightarrow \sigma_2] \cup W^s_0(\sigma_1)$, given $\varepsilon > 0$, $\tau > 0$ and $P \in V^0$, we may find $R \in V^0 \setminus  [\sigma_1 \rightarrow \sigma_2]$ at a distance from $P$ smaller than $\varepsilon$ and whose $\omega$--limit by the flow of $f_0$ contains $W^s_0(\sigma_1)$. Thus, there is $t > \tau$ such that the distance from $\varphi_\lambda(t,R)$ to $W^s_0(\sigma_1)$ is smaller than $\varepsilon$ (that is, we may find $Q \in W^s_0(\sigma_1)$ such that the distance from $\varphi_\lambda(t,R)$ to $Q$ is smaller than $\varepsilon$). Therefore, $W^s_0(\sigma_1)$ is $(\varepsilon,\tau)$--chain-accessible from $P$ through the $(\varepsilon,\tau)$--trajectory $\{R\}$ that connects $P$ to $Q$. As $\varepsilon$ is arbitrary, this proves that $W^s_0(\sigma_1)$ is chain-accessible from any $P \in V^0$.

For $\lambda > 0$, this property may not be valid due to the presence of sinks or other proper attractors in $\mathcal{A}_\lambda$. Meanwhile, we claim that:
\medskip

\noindent \textbf{Claim}: \emph{If either $W^s_\lambda(\sigma_1)$ or $\sigma_2$ is chain-accessible from a point in $V^0$, then so is every element of $\mathcal{B}_\lambda$.}
\medskip

\begin{proof} Consider $X \in V^0$, $P \in \mathcal{B}_\lambda$ and assume that $W^s_\lambda(\sigma_1)$ is chain-accessible from $X$ (the argument is similar if we replace $W^s_\lambda(\sigma_1)$ by $\sigma_2$). Given $\varepsilon>0$ arbitrarily small, there exist $S_1 \in W^s_\lambda(\sigma_1)$, $\tau>0$ and an $(\frac{\varepsilon}{2},\tau)$--trajectory $\{X_1,\cdots,X_m\}$ connecting $X$ to $S_1$. Iterate now $S_1$ a time $t_1 > \tau$ so that the distance from $\varphi_\lambda(t_1,S_1)$ to $\sigma_1$ is smaller than $\frac{\varepsilon}{2}$. Thus, the set $\{X_1,\cdots,X_m, S_1\}$ is an $(\frac{\varepsilon}{2},\tau)$--trajectory connecting $X$ to $\sigma_1$. Afterwards, take a point $S_2 \in [\sigma_1 \rightarrow \sigma_2]$ at a distance smaller than $\frac{\varepsilon}{2}$ from $\sigma_1$; notice that the points $\{X_1,\cdots,X_m, S_1\}$ also form an $(\varepsilon,\tau)$--trajectory connecting $X$ to $S_2$. If we now iterate $S_2$ a time $t_2 > \tau$, we reach a point $\varphi_\lambda(t_2,S_2)$ which is at a distance smaller than $\frac{\varepsilon}{2}$ from $\sigma_2$. Moreover, as $P \in \mathcal{B}_\lambda$, we may find an $(\frac{\varepsilon}{2},\tau)$--trajectory $\{P_1,\cdots,P_k\}$ connecting $\sigma_2$ to $P$. Therefore, the points $\{X_1,\cdots,X_m, S_2, P_1,\cdots,P_k\}$ produce an $(\varepsilon,\tau)$--trajectory connecting $X$ to $P$. This proves that $\mathcal{B}_\lambda$ is chain-accessible from $X$.
\end{proof}

\appendix
\section{Glossary}\label{Definitions}


\subsection{Hyperbolicity}\label{hyperbolic}
Let $M$ be a smooth compact Riemannian manifold, and consider a diffeomorphism $h: M\rightarrow M$ and a compact $h$-invariant set $\mathcal{K} \subset \EU^3$. We say that $\mathcal{K}$ is \emph{uniformly hyperbolic} if there are constants $C>0$ and $0 < \mu < 1$ such that, for every $x \in \mathcal{K}$, there is a splitting of the tangent space $T_x M=E^s_x\oplus E^u_x$ satisfying, for every $n\in \NN$,
$$\|Dh^n(v)\|\leqslant C \mu^n \|v\|\quad \forall\, v\in E^s_x \qquad \text{and} \qquad \|Dh^{-n}(v)\|\leqslant C \mu^n \|v\| \quad \forall \, v\in E^u_x.$$


\subsection{Symmetry}\label{symmetry}
Given a group $\mathcal{G}$ of endomorphisms of $\EU^3$, we say that a one-parameter family of vector fields $(f_\lambda)_{\lambda \,\in \,\RR}$ is symmetric it it satisfies the equivariance assumption
$$f_\lambda(\gamma x)=\gamma f_\lambda(x) \quad \quad \forall \,x \in \EU^3, \quad \forall \,\gamma \in \mathcal{G}\, \text{ and }\, \forall \,\lambda \in \RR.$$

\subsection{Attracting set}\label{attractor}
A subset $A$ of a topological space $\mathcal{M}$ is said to be \emph{attracting} by a flow $\varphi$ if there exists an open set $U \subset \mathcal{M}$ satisfying
$$\varphi(t,U)\subset U \quad \quad \forall \,t \geqslant 0 \, \text{ and }\, \bigcap_{t\,\in\,\RR^+}\,\varphi(t,U)=A.$$
Its basin of attraction, denoted by $\mathcal{B}(A)$, is the set of points in $\mathcal{M}$ whose orbits have $\omega-$limit in $A$. We say that $A$ is \emph{asymptotically stable} (or, equivalently, that $A$ is a \emph{global attracting set}) if $\mathcal{B}(A) = \mathcal{M}$. An attracting set is said to be \emph{quasi-stochastic} if it encloses periodic trajectories with different Morse indices (the \emph{Morse index} of a hyperbolic equilibrium is the dimension of its unstable manifold), structurally unstable cycles, sinks and saddle-type invariant sets.

\subsection{Heteroclinic phenomena}\label{het_phenomena}
Suppose that $\sigma_1$ and $\sigma_2$ are two hyperbolic saddle-foci of a vector field $f$ with different Morse indices. We say that there is a {\em heteroclinic cycle} associated to $\sigma_1$ and $\sigma_2$ if $W^u(\sigma_1)\cap W^s(\sigma_2)\neq \emptyset$ and  $W^u(\sigma_2)\cap W^s(\sigma_1)\neq \emptyset.$ For $i, j \in \{1,2\}$, the non-empty intersection of $W^u(\sigma_i)$ with $W^s(\sigma_j)$ is called a \emph{heteroclinic connection} between $\sigma_i$ and $\sigma_j$, and will be denoted by $[\sigma_i \rightarrow \sigma_j]$. Although heteroclinic cycles involving equilibria are not a generic feature within differential equations, they may be structurally stable within families of vector fields which are equivariant under the action of a compact Lie group $\mathcal{G}\subset \textbf{O}(n)$, due to the existence of flow-invariant subspaces \cite{GH}. A \emph{heteroclinic network} is a connected finite union of heteroclinic cycles.

\subsection{Bykov cycle}\label{Bykov_def}
A heteroclinic cycle between two hyperbolic saddle-foci of a vector field $f$ with different Morse indices, where one of the connections is transverse (and so stable under small perturbations) while the other is structurally unstable, is called a \emph{Bykov cycle}. {A \emph{Bykov network} is a connected union} of heteroclinic cycles, not necessarily in finite number.


\subsection{Tubular neighborhood of a Bykov cycle}\label{tubular_def}
Given a Bykov cycle associated to $\sigma_1$ and $\sigma_2$, let $V_1, V_2$ be two small and disjoint cylindrical neighborhoods of $\sigma_1$ and $\sigma_2$, respectively. Consider two local sections transverse to the cycle at two points $p_1$ and $p_2$ in the connections $[\sigma_1\rightarrow\sigma_2]$ and $[\sigma_2\rightarrow\sigma_1]$, respectively, with $p_1, p_2 \not\in V_1\cup V_2$. Saturating the cross sections by the flow, we obtain two flow-invariant tubes joining $V_1$ and $V_2$ which contain the connections in their interior. The union of those tubes with $V_1$ and $V_2$ is called a \emph{tubular neighborhood} of the Bykov cycle.

\subsection{Chirality}\label{chirality_def}
There are two different possibilities for the geometry of a flow around a Bykov cycle $\Gamma$, depending on the direction the trajectories turn around the one-dimensional heteroclinic connection between $\sigma_1$ and $\sigma_2$. Let $V_1$ and $V_2$ be small disjoint neighborhoods of $\sigma_1$ and $\sigma_2$ with disjoint boundaries $\partial V_1$ and $\partial V_2$, respectively. Trajectories starting at $ \In(\sigma_1) \setminus W^s(\sigma_1)$ enter $V_1$ in positive time, then follow the connection from $\sigma_1$ to $\sigma_2$, enter $V_2$, finally leaving $V_2$ at $\partial V_2$. Let $\gamma$ be a piece of trajectory like the one just described from $\In(\sigma_1) \setminus W^s(\sigma_1)$ to $\partial V_2$. Now join its starting point to its end point by a line segment (see Figure 1 of \cite{LR2014}), forming a closed curve that we call the  \emph{loop of $\gamma$}. The Bykov cycle $\Gamma$ and the loop of $\gamma$ are disjoint closed sets. We say that the two saddle-foci $\sigma_1$ and $\sigma_2$ in $\Gamma$ have the same \emph{chirality} if the loop of every such a trajectory is linked to $\Gamma$, in the sense that these two sets cannot be disconnected by an isotopy in $\RR^{4}$.

\subsection{Pulse}\label{pulse}
Let $V_1, V_2$ be two small and disjoint neighborhoods of $\sigma_1$ and $\sigma_2$, respectively, and take $n \in \NN$. A one-dimensional heteroclinic connection from $\sigma_2$ to $\sigma_1$ which, after leaving $V_2$, enters and leaves both $V_1$ and $V_2$ precisely $n$ times is called an \emph{$n$--pulse heteroclinic connection} with respect to $V_1$ and $V_2$, or simply an \emph{$n$--pulse}. A $0$--pulse is a one-dimensional heteroclinic connection from $\sigma_2$ to $\sigma_1$ which, after leaving $V_2$, enters $V_1$ and afterwards stays in this neighborhood.

\subsection{Saturated set and the operator tilde}\label{operador_tilde}
Let $\mathcal{S}$ be a cross-section to a flow $\varphi$ and assume that $\mathcal{S}$ contains a compact set $\mathcal{K}$ invariant by the first return map $\mathcal{R}$ to $\mathcal{S}$. Then the \emph{saturation of $\mathcal{K}$}, we denote by $\widetilde{\mathcal{K}}$, is the flow-invariant set formed by the $\varphi$-trajectories of points of $\mathcal{K}$, that is,  $\Big\{\varphi(t,x)\,:\,t\,\in\,\RR,\,\,x\,\in\, \mathcal{K}\Big\}.$

\subsection{Chain-accessible point}\label{accessibility}
Given $\varepsilon>0$ and $\tau > 0$, an $(\varepsilon,\tau)$-\emph{trajectory} of a flow $\varphi$ is a finite set $\{P_1,\cdots,P_k\}$ such that, for all $i\in \{1, \ldots, k-1 \}$, the point $P_{i+1}$ is at a distance  strictly smaller than $\varepsilon$ from $\varphi(t, P_i)$ for some $t > \tau$. A point $Q$ is said to be $(\varepsilon,\tau)$-\emph{chain-accessible} from a point $P$ if there exists an $(\varepsilon,\tau)$-trajectory $\{P_1, P_2,\cdots,P_k\}$ \emph{connecting} $P$ with $Q$; in particular, $P$ and $Q$ are at a distance strictly smaller than $\varepsilon$ from $P_1$ and $P_k$, respectively. A point $Q$ is said to be \emph{chain-accessible} from a point $P$ if, for any $\varepsilon > 0$, there exists $\tau > 0$ such that $Q$ is $(\varepsilon,\tau)$--chain-accessible from $P$.

\subsection{Chain-recurrent class}\label{chain-recurrence}
A point $P$ is called \emph{chain-recurrent} for a flow $\varphi$ if it is chain-accessible from $\varphi(t,P)$ for any $t \in \mathbb{R}$. More generally, a set $\mathcal{C}$ is \emph{chain-accessible} from a point $P$ if it contains a point chain-accessible from $P$. A compact invariant set $\mathcal{C}$ is a \emph{chain-recurrent class} if, given any two points in $\mathcal{C}$ and $\varepsilon >0$, we may find an $(\varepsilon,\tau)$-trajectory connecting them entirely contained in $\mathcal{C}$. We observe that, by definition, every point of a chain-recurrent class is chain-recurrent. A set $\mathcal{C}$ is called \emph{chain-stable} if, for any neighborhood $\mathcal{V}$ of $\mathcal{C}$, there exists an $\varepsilon>0$ such that every $(\varepsilon, \tau)$-trajectory connecting any two points of $\mathcal{C}$ is contained in $\mathcal{V}$.

\end{document}